\documentclass[11pt]{amsproc}

\usepackage{amsmath,amsfonts,eucal, amsthm,amssymb,upref,graphics,latexsym,color,endnotes,setspace}

\usepackage[all]{xy}

\theoremstyle{plain}

\theoremstyle{definition}

\swapnumbers \theoremstyle{plain}
\newtheorem{theorem}{Theorem}[section]
\newtheorem{setup}[theorem]{Setup}

\newtheorem{proposition}[theorem]{Proposition}
\newtheorem{lemma}[theorem]{Lemma}
\newtheorem{corollary}[theorem]{Corollary}

\newtheorem{nothing}[theorem]{}
\newtheorem{sub}{}[theorem] 

\newtheorem{subproposition} [sub]{Proposition}

\newtheorem{sublemma}   [sub]{Lemma}

\newtheorem{subdefinition}[sub]{Definition}

\newcommand{\setspec}[2]{\big\{\,#1\, \mid \,#2\, \big\}}
\theoremstyle{definition}
\newtheorem{definition}[theorem]{Definition}

\newtheorem{example}[theorem]{Example}

\newtheorem{nothing*}[theorem]{}
\newtheorem{notation}[theorem]{Notation}
\newtheorem{remark}[theorem]{Remark}
\newcommand{\Spec}{ \operatorname{{\rm Spec}}}
\newcommand{\haut}{ \operatorname{{\rm ht}}}
\newcommand{\trdeg}{ \operatorname{{\rm trdeg}}}

\newcommand{\Supp}{ \operatorname{{\rm Supp}}}

\newcommand{\lnd}{\operatorname{\text{\sc lnd}}}
\newcommand{\klnd}{\operatorname{\text{\sc klnd}}}
\newcommand{\Aut}{ \operatorname{{\rm Aut}}}
\newcommand{\Frac}{ \operatorname{{\rm Frac}}}
\newcommand{\der}{ \operatorname{{\rm Der}}}

\newcommand{\Min}{ \operatorname{{\rm min}}}

\newcommand{\Sing}{ \operatorname{{\rm Sing}}}

\newcommand{\depth}{ \operatorname{{\rm depth}}}

\newcommand{\ak}{ \operatorname{{\rm AK}}}

\newcommand{\ml}{ \operatorname{{\rm  ML}}}

\newcommand{\bk}{{\ensuremath{\rm \bf k}}}

\renewcommand{\epsilon}{\varepsilon}
\renewcommand{\phi}{\varphi}
\renewcommand{\emptyset}{\varnothing}

\title{Singular Points of Affine ML-Surfaces}
\author{Ratnadha Kolhatkar}
\address{Department of Mathematics and Statistics, University of Ottawa, Ottawa, Canada K1N6N5}
\email{rkolh090@uottawa.ca}

\begin{document}
\maketitle
\begin{abstract} We give a geometric proof of the fact  that any affine surface with trivial Makar-Limanov invariant has finitely many singular points. We deduce that a complete intersection surface with trivial Makar-Limanov invariant is normal.
\end{abstract}

\section{Notation and introduction}
\noindent Let us first fix some notation and recall some basic definitions.
Throughout this paper, unless otherwise specified, {\it $\bk$ will always denote a field of characteristic zero}.  A domain means an integral domain. Given a domain $R$, $\Frac R$ denotes the field of fractions of $R$. By $\bk^{[n]}$, we mean the polynomial ring in $n$ variables over $\bk$ and $\Frac(\bk^{[n]})$ will be denoted by $\bk^{(n)}$.
The set of singular points of a variety $X$ will be denoted by $\Sing (X)$.
\begin{definition} {\rm Given a $\bk$-algebra  $B$, a derivation $D: B \rightarrow B$  is \textit{locally nilpotent} if
for each $b\in B$, there exists a  natural number $n$ (depending on
$b$) such that $D^n (b) = 0$. We use the following notations:
\[\der(B)=\setspec{D}{\text{$D$ is a derivation of $B$}}\]
\[\lnd(B)=\setspec{ D\in \der (B)}{\text{$D$ is locally nilpotent}}\]
 \[\klnd(B)=\setspec{\ker D}{D\in \lnd (B), D\neq 0}\]
 \noindent Given a $\bk$-domain B, one defines its {\it Makar-Limanov invariant} by
  \[ \ml (B) = \bigcap_{D\in \lnd(B)} \ker D.\]
  \noindent If $X = \Spec B$ is an affine $\bk$-variety, define $\ml (X) = \ml (B)$. The Makar-Limanov invariant plays an important role in classifying and distinguishing affine varieties.
  We say that $B$ has trivial Makar-Limanov invariant if $\ml(B) = \bk$.
 }\end{definition}
 Affine spaces $\mathbb A^n_{\bk}$ are the simplest examples of  varieties with trivial Makar-Limanov invariant. While it is known that $\mathbb A^1_\bk$ is the only affine curve which has trivial Makar-Limanov invariant, the class of affine surfaces with trivial Makar-Limanov invariant contains
 many more surfaces, some of which are not even normal. (See Example \ref{example}, for instance.)
 \par Let $\mathcal M (\bk)$ denote the class of 2-dimensional
affine $\bk$-domains which have trivial Makar-Limanov invariant.
 We say  that an affine surface $S= \Spec R$ belongs to the
class $\mathcal M (\bk)$ if $R\in\mathcal M (\bk)$. Such a surface $S$ is also
called a {\it $\ml$-surface}.\\
The following question arises naturally: {\it Classify all surfaces in the class $\mathcal M (\bk)$}.
\par In recent years, researchers including Bandman, Daigle,
Dubouloz, Gurjar, Masuda,  Makar-Limanov, Miyanishi, and Russell (see
\cite{ML}, \cite{Dag6}, \cite{Dag7}, \cite{Dub}, \cite{Gur}, \cite{MaMi}) have been actively investigating properties of normal (or smooth) surfaces belonging to the class $\mathcal M (\bk)$. However, it is desirable to understand what happens when we drop the assumption of normality. For instance, is is natural to ask {\it what are all hypersurfaces of the affine space $\mathbb A^3_{\bk}$ with trivial Makar-Limanov invariant}, and it is not a priori clear that all those surfaces are normal: the fact that they are indeed normal is a consequence of the present paper.
\par In this paper, we prove that a surface in the class $\mathcal M(\bk)$ has only finitely many singular points. As an application, we prove that any  complete intersection surface with trivial Makar-Limanov invariant is normal. Note that these results are valid over any field $\bk$ of characteristic zero.
The results of this paper will be used in a joint paper with D. Daigle \cite{DK}, where we classify all hypersurfaces of $\mathbb A^3_\bk$ (more generally, complete intersection surfaces over $\bk$) with trivial Makar-Limanov invariant.
\par To understand the necessity of some of the arguments given in this paper, the reader should keep in mind certain pathologies that occur when $\bk$ is not assumed to be algebraically closed. For instance, surfaces $S= \Spec R$ belonging to $\mathcal M (\bk)$ are not necessarily rational over $\bk$ and may have very few $\bk$-rational points; moreover, if $\bar \bk$ is the algebraic closure of $\bk$, then $\bar \bk \otimes_{\bk} R$ is not necessarily an integral domain.

\bigskip

\noindent  \textbf {Acknowledgements :} The author wishes to thank Professor Daniel Daigle for his many helpful suggestions and his valuable help in preparing this manuscript.
\section{Preliminaries} \noindent In this section, we gather some basic results and known facts.

\begin{nothing} \label{basic}{\rm Suppose that $B$ is a $\bk$-domain, let $D$  be a nonzero
locally nilpotent derivation of $B$, and let $A=\ker D$. The following are well-known definitions and facts about locally nilpotent derivations: \\
(i) $A$ is {\it factorially closed} in $B$ {\rm(}i.e., the conditions \, $x, y \in B\setminus \{0\}$ and $xy \in A$ imply that $x, y \in A${\rm )}. Consequently, $A$ is algebraically closed in $B$.   \\
(ii) \label{localizationlnd}  Consider the multiplicative set $S= A\setminus \{0\}$ of $B$. We can extend $D$  to an element
 $\mathfrak D \in \lnd (S^{-1}B$) defined by $\mathfrak D(\frac{b}{s}) =
\frac{D(b)}{s}$\,\,. It is well-known that $S^{-1} B = (\Frac A)^{[1]}$.\\
(iii) For every $\lambda \in \bk$, the  map
\[e^{\lambda D}: B\rightarrow B, \quad b\mapsto \sum_{n=0}^{\infty}
\lambda^n \frac{D^n (b)}{n!}\] is a $\bk$-algebra automorphism of $B$. \\
 (iv) Let $\pi:\Spec
B \longrightarrow \Spec A$ be the canonical morphism induced by the
inclusion map $A \hookrightarrow B$. Then there exists a nonempty open set $U\subseteq \Spec A$ such that
\[\text{ $\pi^{-1}(\mathfrak p) \cong \mathbb{A}^1_{\kappa(\mathfrak p)}$ for every $\mathfrak p \in U$, where $\kappa(\mathfrak p)$ is
the residue  field $ A_{\mathfrak p}/ \mathfrak p A_{\mathfrak p}$.}\]
Furthermore, if $\bk$ is algebraically closed and $A$ is $\bk$-affine, then
  \[\textit{$\pi^{-1}(\mathfrak m) \cong \mathbb{A}^1_{\kappa(\mathfrak m)}= \mathbb{A}^1_{\bk}$ \,  for every  closed point \, $\mathfrak m$\, of \, $U$.}\]}
\end{nothing}
\noindent \begin{lemma}\label{notapoint}Given an affine $\bk$-surface $X= \Spec B$, let $A_1$ and $A_2$ be two affine subalgebras of $B$ of dimension $1$. Set $Y_i= \Spec A_i$ and  let $\xymatrix{Y_1\ar@{<-}^-{f_1}[r] & \Spec B\ar[r]^{f_2} & Y_2}$ be the canonical morphisms determined by the inclusions $A_i\hookrightarrow B$ {\rm (}for $i=1,2${\rm)}. If $B$ is algebraic over its subalgebra  $\bk[A_1 \cup A_2]$, then
\[E = \setspec {y\in Y_2}{f_1(f_2^{-1} (y))\,  \text{is a point}}\]
is not a dense subset of $ Y_2$, where by ``$y\in Y_2$" we mean that $y$ is a closed point of $Y_2$.
\end{lemma}
\noindent We leave the proof of Lemma \ref{notapoint} to the reader, as it is basic algebraic geometry and is not directly related to the subject matter of this paper.
\begin{definition} A domain $A$ of transcendence degree $1$ over a field $\bk$ is called a {\it polynomial curve} over $\bk$ if it satisfies the following equivalent conditions:\\
(i) $A$ is a subalgebra of $\bk^{[1]}$.\\
(ii) $\Frac A = \bk^{(1)}$ and $A$ has one rational place at infinity.
\end{definition}
\begin{notation}Given a field extension $F/\bk$, let $\mathbb P_{F/\bk}$ be the set of valuation rings $R$ of $F/\bk$ such that $R\neq F$.
\end{notation}
\begin{lemma}\label{polynomial curve} Let $A$ be a $\bk$-domain. If there exists an algebraic extension $\bk'$ of $\bk$ such that $\bk'\otimes_\bk A$
is a polynomial curve over $\bk'$, then $A$ is a polynomial curve over $\bk$.
\end{lemma}
\begin{proof} We sketch a proof of this fact, as we were unable to find a suitable reference.
It is easy to prove that $A$ is affine. We may assume that $[\bk' : \bk] < \infty$. Let $F= \Frac A$ and $F' = \Frac A'$, where $A' = \bk' \otimes_{\bk} A$. Note that $[F' : F] = [\bk' :\bk]$ and $F' = \bk'F$. In the terminology of \cite{S}, the function field $F'/ \bk'$ is an algebraic constant field extension of $F/\bk$. By Theorem III.6.3 of \cite{S}, $F'/\bk'$ has same genus as $F/\bk$ (hence, $F/\bk$ has genus zero) and $F'/ F$ is unramified. It remains to prove that $A$ has one rational place at infinity. Let
\[\text{$E= \setspec{R\in \mathbb P_{F/\bk}}{A\nsubseteq R}$ and
$E'= \setspec{R'\in \mathbb P_{F'/\bk'}}{\bk'\otimes_{\bk} A
\nsubseteq R'}$}.\] If $R$ is any element of $E$, then every $R' \in
\mathbb P_{F'/\bk'}$ lying over $R$ (i.e., satisfying $R'\cap F = R$) must belong to $E'$. But $E'$
is a singleton, say $E' = \{R'\}$. It follows that $E$ is a singleton, say $E= \{R\}$.  Let
$\kappa'$ and $\kappa$ be the residue fields of $R'$ and $R$,
respectively.  Then $[F':F] = e f$, where $f = [\kappa': \kappa]$
and $e$ is the ramification index of $R'$ over $R$. As $F'/F$ is
unramified, we have $e=1$. Since $\bk'\otimes_{\bk} A$ is a polynomial
curve over $\bk'$, $\kappa' = \bk'$. Hence
\[[\bk' : \bk] = [ F' : F] = ef = [ \kappa' : \kappa] = [\bk':
\kappa].\] Thus, $\kappa = \bk$ and $A$ has one rational place at
infinity.
\end{proof}
\noindent The following lemma can be obtained as an easy consequence of \text{\rm\cite[Lemma $3.1$]{Dag5}}.
\begin{lemma}\label{gjvd} Let $B$ be a $\bk$-algebra and $f(T)\in B[T]$, where $T$ is an indeterminate. \\
{\rm (a)} If $f(T)$ has
infinitely many roots in $\bk$, then $f(T)=0$.\\
{\rm (b)} If $J$ is an ideal of $B$ and $f(\lambda)\in J$ for infinitely many $\lambda \in \bk$, then $f(T) \in J[T]$.
\end{lemma}
\begin{definition}\label{intid}Let $R$ be a ring and $D\in \der(R)$. An ideal
 $I$ of $R$ is called an {\it integral ideal} for $D$ if $D(I)\subseteq I$.
\end{definition}
\begin{lemma}\label{opwbdkwed}Let $R$ be a $\bk$-domain, and  let $I$ be a nonzero ideal of $R$. If $A\in
\klnd(R)$, then the following statements are equivalent:\\
{\rm (1)} $I\cap A \neq (0)$.\\
{\rm (2)} There exists $ D\in \lnd(R)$ such that $\ker D =A$ and $I$
is an integral ideal for $D$.
\end{lemma}
\begin{proof} Assume that $(1)$ holds. Let $0\neq a\in I\cap A$, and let $E\in \lnd (R)$ be
such that $A= \ker E$. Since $a\in A$, $aE\in \lnd (R)$ and $aE$ has
kernel $A$. Moreover, as $a\in I$, $(aE)(b)=
a(Eb)\in I$ for all $b\in I$. So $(a E)(I)\subseteq I$, and hence $D := aE$ is the required locally nilpotent
derivation of $R$ proving assertion $(2)$.
\par In the other direction, assume that $0\neq D\in \lnd
(R)$, $\ker D= A$, and $D(I)\subseteq I$.
Choose any $b\in I$, $b\neq 0$. Then the set $\{b, Db, D^2b, \dots\}$ is included in $I$ and contains a nonzero element of $A$.
\end{proof}
\noindent The following is an easy consequence of  \text{\rm\cite[Lemma $2.10$]{D}}.
\begin{lemma} \label{ledkugiq}
 Let $R$ be a noetherian $\bk$-algebra, and let $D\in \der( R)$. If
$I$ is an integral ideal for $D$, so is every minimal prime-over
ideal of $I$.
\end{lemma}
\begin{lemma} \label{Expintegral}Let $B$ be a $\bk$-algebra, $J$ an ideal of $B$, and $D\in \lnd(B)$. If $e^{tD} (J)\subseteq J$ for some nonzero $t\in \bk$, then $J$ is an integral ideal for $D$.
\end{lemma}
\begin{proof} First observe that if $e^{tD} (J)\subseteq J$ for some nonzero $t\in \bk$, then $e^{tD} (J)\subseteq J$ for infinitely many $t\in \bk$. Let $f\in J$. We will show that $D(f)\in J$. Let $n=
\deg_D (f)$, i.e., $n$ is the maximum nonnegative integer such that
$D^n (f)\neq 0$. Define a polynomial $P(T)\in B[T]$ by
\[P(T)= f + D(f) T + \frac {D^2(f)T^2}{2!}+\cdots + \frac {D^n(f)T^n}{n!}.\]
Then for infinitely many $t\in \bk$,
\[P(t)= f + D(f) t + \frac {D^2(f)t^2}{2!}+\cdots + \frac {D^n(f)t^n}{n!}= e^ {tD}(f)\in J.\]
 By Lemma \ref{gjvd}, all coefficients of $P(T)$ belong to $J$, so $D(f)\in J$.
\end{proof}

\begin{lemma}\label{lnd of normalization}Let $B$ be an affine $\bk$-domain, and let $D\in \lnd(B)$.
If $\tilde{B}$ denotes the normalization of $B$, then there exists
$\tilde{D} \in \lnd ( \tilde{B})$ such that $\ker \tilde{D}\cap B =
\ker D$.
\end{lemma}
\begin{proof} We recall the well-known argument. Write $A= \ker D$ and let $S= A \setminus \{0\}$.
Then $D$ extends to a locally nilpotent derivation $\mathfrak D$ of $S^{-1} B$ such that $B\cap \ker \mathfrak D = A$.
As $S^{-1}B$ is a polynomial ring over the field $S^{-1}A$, it is normal, and consequently $B\subseteq \tilde B \subseteq S^{-1}B$.
 It follows that there exists $s\in S$ such that the locally nilpotent derivation $s \,\mathfrak D : S^{-1}B\rightarrow S^{-1}B$ maps $\tilde B$ into itself. The restriction $\tilde D: \tilde B \rightarrow \tilde B$ of $s\, \mathfrak D$ satisfies $\ker \tilde D \cap B = \ker D$.

\end{proof}
\begin{lemma} \label{new} For a two-dimensional affine $\bk$-domain $R$,
\[\text{$|\klnd (R)|> 1$ if and only if\,  $\ml(R)$ is algebraic over $\bk$.}\]
\end{lemma}
\begin{proof} Assume that $\ml (R)$ is algebraic over $\bk$. Since $\trdeg_\bk A = 1$ for any  $A\in \klnd (R)$, it follows that $|\klnd (R)|> 1$.  In the other direction, let  $A$ and $ A'$ be distinct elements of $\klnd (R)$.  As
 $\trdeg_\bk  A= 1 = \trdeg_\bk A'$ and $ A\cap A'$ is algebraically closed in $R$, it follows that $ A\cap  A'$ is algebraic over $\bk$. Hence $\ml(R)$ is algebraic over $\bk$.
\end{proof}
\begin{corollary}\label{nn} If  $R\in \mathcal M (\bk)$, then $\tilde R \in \mathcal M (\bk')$ for some algebraic field extension $\bk'\supseteq \bk$ such that $\bk' \subset \tilde R$. In particular, if $\bk$ is algebraically closed, then $\ml(\tilde R)= \bk$.
\end{corollary}

\begin{proof} As $R\in \mathcal M (\bk)$, we get $|\klnd (R)|>1$ by Lemma \ref{new}.
Let $A_1$ and $A_2$ be distinct elements of $\klnd(R)$. There exist $\tilde A_1, \tilde A_2 \in \klnd (\tilde R)$ satisfying $\tilde A_i \cap R = A_i$ (cf. \ref{lnd of normalization}), so $|\klnd (\tilde R)|>1$. Hence $\ml (\tilde R)$ is algebraic over $\bk$ and is a field, say, $\ml (\tilde R) = \bk'$. Then clearly, $\bk\subseteq \bk' \subset \tilde R$ and $\bk'$ is algebraic over $\bk$.
\end{proof}
 \begin{lemma}\label{cute} Let $B\in \mathcal M (\bk)$. If $B$ is normal and $A\in \klnd (B)$, then $A\cong \bk^{[1]}$.
\end{lemma}
\begin{proof} This result is well-known when $\bk$ is algebraically closed. (See 2.3 of \cite{Dag7}, for instance.) To prove the general case,  denote the algebraic closure of $\bk$ by $\bar \bk$.
Let $A\in \klnd (B)$ and note that $A$ is a $1$-dimensional noetherian normal domain. To prove that $A\cong \bk^{[1]}$, it suffices to check that $A \subseteq \bk^{[1]}$.
 By Lemma $3.7$ of \cite{Dag6}, $\mathcal B:=\bar \bk \otimes_{\bk} B$ is an integral domain and $\ml (\mathcal B) = \bar \bk$. If $\tilde {\mathcal B}$ denotes the normalization of $\mathcal B$, then
$\ml (\tilde {\mathcal B}) = \bar \bk$ by Corollary \ref{nn}. Note that each element of $\klnd (\tilde {\mathcal B})$ is isomorphic to ${\bar \bk}^{[1]}$. Given $A\in \klnd (B)$, $\bar \bk \otimes_{\bk} A \in \klnd (\mathcal B)$ and there exists $D\in \lnd (\tilde{\mathcal B})$ such that $\ker D \cap  {\mathcal B}= \bar \bk \otimes_{\bk} A$ (cf. Lemma \ref{lnd of normalization}). As $\ker D \cong {\bar \bk}^{[1]}$, it follows that $\bar \bk \otimes_{\bk} A \subseteq {\bar\bk}^{[1]}$. Then $A\subseteq \bk^{[1]}$ by Lemma \ref{polynomial curve}.
\end{proof}
\section  {Completion of surfaces and fibrations}\label{3}
Throughout Section \ref{3}, we fix $\bk$ to be an algebraically closed
field of characteristic zero. All varieties  are assumed to be
$\bk$-varieties. In this section, we state some
properties of affine normal surfaces, fibrations on such surfaces,
 and completions of such surfaces. The material of this section is well-known.
\begin{nothing}{\rm Let $S$ be a complete normal surface. By an {\it SNC-divisor} on $S$, we mean a Weil divisor $D= \sum_{i=1}^{n} C_i$ where $C_1, \dots, C_n$ are distinct irreducible curves on $S$ satisfying the following conditions: \\
 (i) $\Supp (D)= \bigcup_{i=1}^{n} C_i$ is included in $S\setminus \Sing (S)$.\\
 (ii) Each irreducible component $C_i$ of $D$ is isomorphic to $\mathbb P^1$.\\
 (iii) If $i\neq j$ then $C_i\cdot C_j
 \leq 1$.\\
 (iv) If $i, j, k$ are distinct then $C_i\cap
 C_j\cap C_k = \emptyset$.}
\end{nothing}
\begin{definition} An {\it $\mathbb A^1$-fibration} (respectively, a {\it $\mathbb
P^1$-fibration}) on a surface $S$ is a surjective morphism
$\rho: S\rightarrow  Z$ on a nonsingular curve $Z$ whose
general fibres  are isomorphic to $\mathbb A^1$ (respectively, to
$\mathbb P^1$).  For our purposes, we will always consider $\mathbb
A^1$-fibrations whose codomain $Z$ is  $\mathbb
A^1$.
\end{definition}

\begin{definition} \label{sense}Let $S$ be an affine normal surface and $\rho: S\rightarrow \mathbb A^1$ an $\mathbb A^1$-fibration. By a {\it completion of the pair $(S, \rho)$}, we mean a commutative diagram of morphisms of algebraic varieties
\begin{equation} \label{completion}
 \xymatrix{
 S\ar[d]_{\rho}\ar@{^{(}->}[r] & \bar{S}\ar[d]^{\bar{\rho}}\\
 \mathbb{A}^{1}\ar@{^{(}->}[r] & \mathbb{P}^{1}}
\end{equation}
such that the ``$\hookrightarrow$" are open immersions, $\bar S$ is a complete normal surface, and $\bar S \setminus S$ is the support of an SNC-divisor of $\bar S$. \end{definition}
\noindent It is well-known that given any affine normal surface $S$ and an $\mathbb A^1$-fibration $\rho: S \rightarrow \mathbb A^1$, there exists a completion of $(S, \rho)$.

\begin{setup}\label{stup}
{\rm Throughout Paragraph \ref{stup}, we assume:\\
{\rm (i)} $S$ is an affine normal surface.\\
{\rm (ii)} $\rho: S\rightarrow \mathbb A^1$ is an $\mathbb A^1$-fibration.\\
{ \rm (iii)} $(\bar S, \bar \rho)$ is a completion of $(S, \rho)$, with notation as in Diagram (\ref{completion}); we let $D$ be the SNC-divisor of $\bar S$ whose support is $\bar S\setminus S$.}
\end{setup}

\noindent As  $\bar S$ is complete, $\bar \rho$ is closed. So given any curve $C\subset \bar S$, $\bar \rho (C)$ is either a point or all of $\mathbb P^1$. Accordingly we have:
\begin{subdefinition}{\rm A curve $C\subset \bar{S}$ is said to be
{$\bar{\rho}$-\it vertical} if $\bar{\rho}(C)$ is a point. Otherwise,
we say that the curve is {$\bar{\rho}$-\it horizontal}. Thus
$C\subset \bar{S}$ is $\bar{\rho}$-horizontal if and only if $\bar{\rho}(C) =
\mathbb{P}^1$.}
\end{subdefinition}
\begin{sublemma} \label{cv}Let the setup be as in \ref{stup}.
\begin{enumerate}\item[{\rm(a)}] For a general point $z\in \mathbb P^1$, $\bar \rho^{-1} (z)\cong \mathbb P^1$ and $\bar \rho^{-1} (z) \cap S \cong \mathbb A^1$. In particular, $\bar \rho: \bar S \rightarrow \mathbb P^1$ is a $\mathbb P^1$-fibration.
\item [{\rm (b)}] Exactly one irreducible component of $D$ is $\bar \rho$-horizontal.
\end{enumerate}
\end{sublemma}
\begin{proof}As these facts are well-known, we only sketch the proof. By commutativity of Diagram (\ref{completion}), $\bar \rho^{-1} (z) \cap S = \rho^{-1} (z) \cong \mathbb A^1$ for general $z\in \mathbb P^1$. Assertion (a) follows from this. It also follows that the general fibre $\bar \rho^{-1} (z)$ meets $D$ in {\it exactly} one point, and this implies that $D$ has exactly one horizontal component.
\end{proof}

\section{Geometry of surfaces in the class $\mathcal M (\bk)$} \label{4} In this section, $\bk$ is an arbitrary field of characteristic zero (except in \ref{4.1} and \ref{mb}, where it is assumed to be algebraically closed).
\begin{setup}\label{4.1} {\rm The following assumptions and notations are valid throughout Paragraph \ref{4.1}. Suppose that $\bk$ is algebraically closed. Fix
$B\in \mathcal{M}(\bk)$, suppose that $B$ is normal, and let $S= \Spec B$. Consider distinct elements $A_1, A_2 \in \klnd (B)$ and recall from \ref{cute} that $A_i \cong \bk^{[1]}$ for $i=1,2$. Let $\rho_i : S \rightarrow \mathbb A^1$ be the morphism determined by the inclusion $A_i\hookrightarrow B$ for $i=1,2$. It follows from \ref{basic}(iv) that $\rho_1$ and $\rho_2$ are $\mathbb A^1$-fibrations, and \ref{notapoint} implies that $\rho_1$ and $\rho_2$ have distinct general fibres. Choose a complete normal surface $\bar S$ and morphisms $\bar \rho_1, \bar \rho_2: \bar S \rightarrow \mathbb P^1$ such that, for each $i=1,2$, $(\bar S, \bar \rho_i)$ is a completion of $(S, \rho_i)$ in the sense of \ref{sense}. We also consider  the following
diagram:
\begin{equation}\label{16}
 \xymatrix{
 S\ar@/^/[d]^{\rho_2}\ar@/_/[d]_{\rho_1}\ar@{^{(}->}[r] & \bar{S}\ar@/^/[d]^{\bar {\rho_2}}\ar@/_/[d]_{\bar \rho_1}\\
 \mathbb{A}^{1}\ar@{^{(}->}[r] & \mathbb{P}^{1}}
\end{equation}
Let $\infty$ be such that $\mathbb P^1 = \mathbb A^1 \cup \{\infty\}$ in Diagram (\ref{16}). For $i=1, 2$, let $H_i$ be the unique irreducible component of $D=\bar S \setminus S$ which is $\bar \rho_i$-horizontal. (See Lemma \ref{cv}.)

}
\end{setup}

\begin{sublemma}\label{4-1}We have
 $\bar{\rho_1}(H_2) = \{\infty\}$ and $\bar{\rho_2}(H_1)
= \{\infty\}$. In particular, $H_1\neq H_2$.
\end{sublemma}
\begin{proof}Recall that $H_i\subseteq D$ and $\bar{\rho_i}(H_i)= \mathbb{P}^1$ for each $i=1,2$.  For a general $z_1\in
\mathbb{P}^1$, $(\bar{\rho_1})^{-1} (z_1) = C_1$, where $C_1$ is an irreducible curve of $\bar S$ which intersects $H_1$ in a unique point, say $Q$.  As  $\rho_1$ and $\rho_2$ have distinct general fibres, we choose $z_1$ so that
$\rho_2 (\rho_1^{-1}({z_1}))$ is not a point. Then $ \bar \rho_2 (C_1)$ is not a point,
so $\bar{\rho_2}(C_1)= \mathbb{P}^1$. Choose $Q_1\in C_1$ such that
$\bar{\rho_2}(Q_1)= \{\infty\}$. Clearly, $Q_1\in D$. Since $C_1$ meets $D$ in exactly one point, $C_1\cap D = \{Q_1\}$. Consequently, $\{Q \}=C_1\cap H_1 \subseteq C_1\cap D = \{Q_1\}$.
It follows that $\{Q_1\} = {C_1\cap H_1}$. Repeating this
process for infinitely many points $z_i$ of
$\mathbb{P}^1$, we get infinitely many points $Q_i \in
H_1$ satisfying $\bar{\rho_1} (Q_i) = z_i$ and $\bar{\rho_2}(Q_i) =
\{\infty\}$.  Hence we conclude that $\bar{\rho_2}(H_1)=\{\infty\}$.
Similarly, we can prove that $\bar{\rho_1}(H_2)=\{\infty\}$. As $\bar \rho_1 (H_1) = \mathbb P^1 = \bar \rho_2(H_2)$, it follows immediately that  $H_1$ and $H_2$ are distinct.
\end{proof}
\begin{subproposition}\label{curve shrinking to points} There  does not exist
an irreducible curve $C\subset S$ such that
$\rho_1(C)$ and $\rho_2(C)$ are points.
\end{subproposition}
\begin{proof}
By contradiction, suppose that there exists an irreducible curve
$ C_0$ of $S$ such that $\rho_1(C_0) = a_1$ and $\rho_2(C_0) = a_2$ for
some points $a_i\in\mathbb A^1$.
Consider $C := \bar{C_0}$, the
closure of $C_0$ in $\bar{S}$. Then $C$ is a curve in $\bar{S}$ such
that $C\cap D\neq \emptyset$, $\bar \rho_1 (C) = a_1$, and $\bar \rho_2 (C) = a_2$ (where $a_1, a_2 \in \mathbb P^1\setminus \{\infty\}$). Since $D$ is connected, there is an integer $k\geq 1$ and a sequence $D_1, \dots, D_k$ of irreducible components of $D$ satisfying: \\
$\bullet$ For each $1\leq i < k$, $D_i$ is $\bar \rho_1$-vertical and $\bar \rho_2$-vertical, and $D_k \in \{H_1, H_2\}$. \\
$\bullet$ $C\cap D_1 \neq \emptyset$, and $D_i \cap D_{i+1} \neq \emptyset$ (for $1\leq i < k$).
\par Note that $\bar \rho_j (D_k) = \infty$ for some $j\in \{1,2\}$.
Since $C\cup D_1\cup \cdots \cup D_k$ is connected, it follows that $\bar \rho_j (C\cup D_1\cup \cdots \cup D_k)$ is connected and is a finite set of points, i.e.,
is one point. But $a_j , \infty \in \bar \rho_j (C\cup D_1\cup \cdots \cup D_k)$, so we obtain a contradiction.

\end{proof}
\noindent For the remainder of this paper, we assume that $\bk$ is an arbitrary field of characteristic zero.
\begin{definition} Let $B$ be an integral domain of characteristic zero. We say that $B$ has property $(\ast)$ if $B$ has no height $1$ proper ideal $I$
which intersects two distinct elements $A_1, A_2 \in \klnd(B)$ nontrivially. That is, $B$ has property $(\ast)$ if $I\cap A_1= 0$ or $I\cap A_2=0$ for all height $1$ proper ideals $I$ of $B$ and all distinct $A_1, A_2\in \klnd(B)$.
\end{definition}
\smallskip\noindent  Our next goal is to prove Theorem \ref{satisfies *}.  We do this in several steps, as follows.
\begin{corollary} \label{mb}Suppose that $\bk$ is algebraically closed and that $B\in \mathcal {M}(\bk)$ is normal. Then $B$ has property $(\ast)$.
\end{corollary}
\begin{proof} By contradiction, suppose that there exist distinct $A_1, A_2\in \klnd(B)$ and a height $1$ ideal $I$ of $B$ such that $I\cap A_i \neq 0$ for $i= 1, 2$. Pick a height $1$ prime ideal $\mathfrak p$ of $B$ such that $\mathfrak p \supseteq I$, and note that $\mathfrak p \cap A_i \neq 0$ for $i=1,2$. So the irreducible curve $C= V(\mathfrak p)\subset \Spec B$ is mapped to a point by each canonical morphism  $\rho_i: \Spec B\rightarrow  \Spec A_i$ ($i= 1,2$). This contradicts \ref{curve
shrinking to points}.

\end{proof}
\begin{notation} Let $B\subseteq B'$ be integral domains of characteristic zero. We write $B\triangleleft B'$ to indicate that $B'$ is integral over $B$ and that, for each $A\in \klnd(B)$, there exists $A'\in \klnd(B')$ such that $A'\cap B = A$. Clearly, $\triangleleft$ is a transitive relation.
\end{notation}
\begin{lemma}\label{ppty}Let $B, B'$ be integral domains of characteristic zero such that $B\triangleleft B'$. If $B'$ has property {\rm ($\ast$)}, then so does $B$.
\end{lemma}
\begin{proof} Let $I\neq B$ be a height $1$ ideal of $B$ and let $A_1, A_2 \in \klnd(B)$ satisfy $I\cap A_i \neq 0$. As $B'$ is integral over $B$, $IB' \neq B'$ and
 $\haut IB' = 1$.
Since $B\triangleleft B'$, there exist $A_1', A_2' \in \klnd (B')$ such that $A_i' \cap B = A_i$ for $i=1,2$. Moreover, $A_i' \cap IB' \supset A_i \cap I \neq 0$. Since $B'$ has
property ($\ast$), it follows that $A_1' = A_2'$. Consequently, $A_1= A_2$.
\end{proof}
\noindent Recall that $\bk$ is an arbitrary field of characteristic zero.
\begin{theorem}\label{satisfies *} Each element $B$ of $\mathcal {M}(\bk)$ has property {\rm($\ast$)}.
\end{theorem}
\begin{proof} If $\tilde B$ denotes the normalization of $B$, $B\triangleleft \tilde B$ follows by Lemma \ref{lnd of normalization}. Moreover,  Corollary \ref{nn} implies that $\tilde B\in \mathcal M (\bk')$ for some field $\bk'$. As  $B\triangleleft \tilde B$, it suffices to prove the theorem when $B$ is normal by Lemma \ref{ppty}.
 \par If $B$ is normal, $\mathcal B =\bar \bk \otimes_{\bk} B$ is an integral domain and $\ml (\mathcal B) = \bar \bk$ by Lemma $3.7$ of \cite{Dag6}. Then the normalization $\tilde {\mathcal B} \in \mathcal {M} (\bar \bk)$ by Corollary \ref{nn}, so $\tilde {\mathcal B}$ has property ($\ast$) by \ref{mb}.
  It suffices to prove that $B\triangleleft \tilde {\mathcal B}$ because then the result follows by Lemma \ref{ppty}.
  \par As $\bar \bk$ is integral over $\bk$, it follows that
   $\bar \bk \otimes_{\bk} B$ is integral over $\bk\otimes_{\bk} B \cong B$. Furthermore, given $A\in \klnd (B)$, $\bar A =\bar \bk \otimes _{\bk} A$ belongs to $\klnd (\mathcal B)$ and satisfies $\bar A \cap (\bk\otimes_{\bk} B) =  A$. This proves that $B\triangleleft \mathcal B$. Finally, $\mathcal B\triangleleft \tilde {\mathcal B}$ and  $\triangleleft$ is transitive, so it follows that $B\triangleleft \tilde {\mathcal B}$.
  \end{proof}

  \begin{remark}Every two-dimensional affine $\bk$-domain has property $(*)$.
Indeed, let $B$ be such a ring.  If $| \klnd(B) | \le 1$, then it
is trivial that $B$ has property $(*)$.
If $| \klnd(B) | > 1$ then
$B \in \mathcal M( \bk' )$ for some field $\bk'$, where $\bk'$ is algebraic over $\bk$ (cf. Lemma \ref{new}). Then the result follows from Theorem \ref{satisfies *}
\end{remark}
  \begin{definition} An affine scheme $\Spec A$ is {\it regular in codimension $1$} if and only if $A_{\mathfrak p}$ is regular for every height 1 prime ideal
$\mathfrak p$ of $A$.
\end{definition}
\begin{theorem}\text{\rm\cite[Thm $73$, p.246]{Matoldbook}} Let $A$ an affine domain containing a field. Then
 \[U = \setspec{ \mathfrak p \in \Spec A}{\text{$A_{\mathfrak p}$ is  a regular local ring}}\]
is a nonempty open subset of the affine scheme $X= \Spec A$.
\end{theorem}
\begin{proposition}\label{good result}Let $B$ be an affine $\bk$-domain. If $\mathfrak p$ is a height $1$ prime ideal of $B$ such that $B_\mathfrak p$ is not regular, then $D(\mathfrak p) \subseteq \mathfrak p$ for every $D\in \lnd(B)$.
\end{proposition}
\begin{proof} The set $T= \setspec{ \mathfrak p \in \Spec B}{\text{$B_{\mathfrak p}$ is  not regular}}$ is a closed and proper subset of $X:= \Spec B$.
For every $\mathfrak p\in T$ satisfying $\haut \mathfrak p =1$, the closure  $\overline{\{\mathfrak p\}}$ is an irreducible component of $T$ and $\mathfrak p$ is the unique generic point of that component.
 As $T$ has only finitely many irreducible components, it follows that $T$ contains only finitely many prime ideals of height $1$. Denote these prime ideals by $\mathfrak p_1, \dots , \mathfrak p_n$.
\par Pick  $\mathfrak p \in \{\mathfrak p_1, \cdots , \mathfrak p_n\}$ and $D\in \lnd(B)$. We will prove that $D(\mathfrak p) \subseteq \mathfrak p$. In view of Lemma \ref{Expintegral}, it is enough to show that
\begin{equation}\label{6}\text{ $e^{\lambda D} (\mathfrak p)\subseteq \mathfrak p$ for some nonzero $\lambda\in \bk$. }
\end{equation}
As the group $\Aut (B)$ acts on the set $T$, it follows that it acts on $\{\mathfrak p_1, \dots, \mathfrak p_n\}$. Furthermore,
$\bk = \bigcup_{i=1}^n \setspec{\lambda \in \bk}{e^{\lambda D} (\mathfrak p) =\mathfrak p_i}$ . Since $\bk$ is infinite, there exists $i\in\{1, \dots, n\}$ such that
$\Omega:= \setspec{\lambda \in \bk}{e^{\lambda D} (\mathfrak p) = \mathfrak p_i}$ is infinite. Pick distinct elements $\lambda_1, \lambda_2$ of $\Omega$. Then
$e^{(-\lambda_2+\lambda_1)D} (\mathfrak p) \subseteq \mathfrak p$. So (\ref{6}) is true.
\end{proof}
\begin{corollary} \label{at last}If $B\in \mathcal M (\bk)$ and $X=\Spec B$, then the set
\[\Sing (X) = \setspec{\mathfrak p\in \Spec B}{\text{$B_\mathfrak p$ is not a regular local ring} }\]
is finite. Consequently, $B$ is regular in codimension $1$.
\end{corollary}
\begin{proof} The set $T= \Sing X$ is a proper closed subset of $X$, so $\dim T \leq 1$.
It follows by \ref{good result} that given a height $1$ prime ideal $\mathfrak p$ of $B$ belonging to $T$, $D(\mathfrak p) \subseteq \mathfrak p$ for every $D\in \lnd (B)$. Then \ref{opwbdkwed} implies that $\mathfrak p\cap \ker D\neq 0$ for every $D\in \lnd (B)$. Since $B$ has property ($\ast$) by \ref{satisfies *}, we obtain that the set $\klnd(B)$ is a singleton, a contradiction. So $T$ contains no height $1$ prime ideal; consequently, $B$ is regular in codimension $1$. This also proves that $\dim T = 0$. So $T$ is a finite set of maximal ideals.
\end{proof}

\section {An application to complete intersections}
\begin{definition} \label {ci} Let $A$ be a domain containing a field $\bk$. We say that $A$ is a {\it complete intersection over $\bk$} if it is isomorphic to a quotient
\[ \bk[X_1, \dots, X_n]/ (f_1, \dots, f_p)\]
for some $n, p \in \mathbb N$, where $(f_1, \dots, f_p)$ is a prime ideal of $\bk[X_1, \dots, X_n]$ of height $p$. If $R$ is a complete intersection over $\bk$, we also call $\Spec R$ a complete intersection over $\bk$.
\end{definition}
\noindent Recall the following criterion for noetherian normal rings due to Serre.
\begin{theorem}{\rm (Serre) }A noetherian ring $A$ is normal if and only if it satisfies\\
$(R_1)$ $A_{\mathfrak p}$ is regular for all $\mathfrak p \in \Spec A$ with $\haut \mathfrak p \leq 1$, and \\
$(S_2)$  $\depth A_{\mathfrak p} \geq \Min (\haut \mathfrak p, 2)$ for all $\mathfrak p \in \Spec A$.
\end{theorem}
\begin{corollary}\label{vapara2}Let $B\in \mathcal M (\bk)$. If $B$ satisfies Serre's condition $(S_2)$, then $B$ is normal. In particular, complete intersection surfaces in the class $\mathcal M (\bk)$ are normal.
\end{corollary}
\begin{proof} Consider $B\in \mathcal M (\bk)$ and suppose that $B$ satisfies $(S_2)$. To show that $B$ is normal, it suffices to prove that $B$ satisfies $(R_1)$. So let $\mathfrak p \in \Spec B$. If $\haut \mathfrak p = 0$, then clearly $B_{\mathfrak p}$ is regular. If $\haut \mathfrak p =1$, $B_\mathfrak p$ is regular by Corollary \ref{at last}.
\par If $B$ is a complete intersection, then $B$ is Cohen-Macaulay (cf. \text{\rm\cite[Prop.$18.13$]{Eis}}), and so it satisfies $(S_2)$ (cf. \text{\rm\cite[17.I, p.125]{Matoldbook}}). Then the result follows by the previous case.
\end{proof}
\begin{example}\label{example}Let $B= \bk[ x, xy, y^2, y^3]$. Then $D= x
\frac{\partial}{\partial y}$, $E=y^2 \frac{\partial}{\partial x}$
are two nonzero locally nilpotent derivations of $B$ and $\ml(B)=\bk$. Note that $B$ is not normal. So by Corollary \ref{vapara2},  $\Spec B$ is not a complete intersection surface over $\bk$. By similar arguments, we can prove that $ S:= \Spec \bk[ x^2, x^3, y^3, y^4, y^5, xy, x^2y, xy^2, xy^3]$ is a ML-surface which is not a complete intersection surface over $\bk$.
\end{example}

 \end{document}